\documentclass[a4paper,10pt]{amsart}
\usepackage{amsmath}
\usepackage[T1]{fontenc}
\usepackage{amssymb}
\usepackage{amsthm}
\newcommand{\md}{{\rm d}}

\newtheorem{thm}{Theorem}[section]

\newtheorem{lem}[thm]{Lemma}

\newtheorem{rem}[thm]{Remark}
\newtheorem{conj}[thm]{Conjecture}
\numberwithin{equation}{section}

\title[The energy space for the G-P equation]{The energy space for the Gross-Pitaevskii equation with magnetic field}

\author[A. Kachmar]{Ayman Kachmar}
\address{A. Kachmar\newline
Universit{\'e} Paris-Sud\\ D{\'e}partement de
math{\'e}matique\\B{\^a}t. 425\\F-91405 Orsay}
\email{ayman.kachmar@math.u-psud.fr}

\begin{document}
\begin{abstract}
We study the energy space for  the Gross-Pitaevskii equation with
magnetic field and non-vanishing conditions at infinity. We provide
necessary and sufficient conditions on the magnetic field for which
the energy space is non-empty.
\end{abstract}
\maketitle
\section{Introduction}
Let us consider  the Gross-Pitaevskii  equation with magnetic field,
\begin{equation}\label{eq:GL}
i\partial_t\psi=(\nabla-iA)^2\psi+(1-|\psi|^2)\psi\quad{\rm
in~}\mathbb R\times\mathbb R^2\,,
\end{equation}
where $\psi$ is a complex-valued wave function and $A\in C^2(\mathbb
R^2;\mathbb R^2)$ is a given magnetic potential - the magnetic field
being,
\begin{equation}\label{eq:mf}
B={\rm curl}\,A\,. \end{equation} Setting $A=0$, we get the usual
Gross-Pitaevskii equation, which is being  intensively studied, see
for instance the papers \cite{BS, Ga, G}, where solutions with
non-vanishing conditions at infinity appear to be of particular
interest.\\
 When seeking solutions of (\ref{eq:GL}) with non-vanishing
conditions at infinity, the natural set-up is to work in the energy
space,
\begin{equation}\label{eq:GL-sp}
\mathcal E_B=\left\{\psi\in H^1_{\rm loc}(\mathbb R^2;\mathbb
C)~:~(\nabla-iA)\psi\,,~1-|\psi|^2\in L^2(\mathbb R^2)\right\}\,.
\end{equation}
The rough justification is that Eq. (\ref{eq:GL}) appears formally
as the Hamiltonian evolution of the Ginzburg-Landau
energy,\footnote{We may some times write $E_B(\psi,A)$ instead of
$E_B(\psi)$, in order to point out the dependence on $A$.}
\begin{equation}\label{eq:en}
E_B(\psi)=\frac12\int_{\mathbb
R^2}\bigg{(}|(\nabla-iA)\psi|^2+\frac12(1-|\psi|^2)^2\bigg{)}\,\md
x\,.
\end{equation}
In the presence of magnetic fields, i.e. when $B$ does not vanish,
it is far from obvious that the energy space (\ref{eq:GL-sp}) is
non-empty for any magnetic potential $A$. As we shall see, this will
be entirely dependent on the magnetic field $B$ (for instance, when
$B$ is constant, $\mathcal E_B$ will be empty).

\begin{thm}\label{thm:1}
Assume that the magnetic field satisfies
\begin{equation}\label{eq:hyp-B}
B\in C^1(\mathbb R^2;\mathbb R)\cap L^\infty(\mathbb R^2;\mathbb
R)\,,
\quad B(x)\geq0\quad\forall~x\in\mathbb R^2\,,
\end{equation}
and let $A$ be any magnetic potential satisfying (\ref{eq:mf}). Then
the energy space $\mathcal E_B$ is non-empty if and only if
$B\in L^1(\mathbb R^2;\mathbb R)$. 
\end{thm}

\begin{rem}\label{rem:1}
{\rm \begin{enumerate}
\item We drop the  magnetic potential $A$ from the notation due
to gauge invariance. Actually, if $\psi\in H^1_{\rm loc}(\mathbb
R^2)$ is such that $E_B(\psi,A)<\infty$, then for all $\chi\in
H^1_{\rm loc}(\mathbb R^2)$, $E_B(\psi
e^{i\chi},A+\nabla\chi)<\infty$.
\item Thanks to gauge invariance, we may always assume, under the assumptions made in
Theorem~\ref{thm:1}, that $A\in C^2(\mathbb R^2)$.
\item If one may pick a potential $A'\in L^2(\mathbb R^2)$ such that
$ {\rm curl}\,A'=B$, then it is clear that the energy space
$\mathcal E_B$ is non-empty, as it contains  a function of constant
module, $e^{i\chi}$. Actually this will be shown to be the case if
we assume, in addition to the hypotheses made in
Theorem~\ref{thm:1}, that $B\in L^1(\mathbb R^2;\mathbb R)$.
\item As an immediate corollary of Theorem~\ref{thm:1}, if the
magnetic field is constant or more generally if
$$B(x)\to c\quad{\rm as~}|x|\to\infty\,,\quad c>0\,,$$
then the energy space $\mathcal E_B$ is empty.
\item The hypotheses on the sign of $B$ is to establish the necessary condition. As one may check through the
proof, this can be relaxed to $B$ of constant sign.
 \end{enumerate}}
\end{rem}

The hypotheses made in Theorem~\ref{thm:1}  on the magnetic field
$B$ are physically relevant and fit the regimes observed in the
analysis of the Ginzburg-Landau functional, as one might see the
books \cite{FH, SaSe}. However, we may give further generalizations
when dropping the hypothesis that the magnetic field is bounded and
positive, as we indeed do in the next two theorems.

\begin{thm}\label{thm:2}
Assume that the magnetic field  satisfies $B\in C^1(\mathbb
R^2;\mathbb R)$ and  $B={\rm curl}\,A$ for some $A\in
L^\infty(\mathbb R^2;\mathbb R^2)$ such that ${\rm div}\,A\in
L^\infty(\mathbb R^2;\mathbb R)$.\\
Then the energy space $\mathcal E_B$ is non-empty if and only if
$B={\rm curl}\, A'$ for some $A'\in L^2(\mathbb R^2;\mathbb R^2)$.
\end{thm}

\begin{thm}\label{thm:3}
Assume that the magnetic potential satisfies,
$$A\in C^1(\mathbb R^2)\,,\quad \nabla A\in
L^\infty(\mathbb R^2)\,.$$ Then the energy space $\mathcal E_B$ is
non-empty if and only if $B={\rm curl}\, A'$ for some $A'\in
L^2(\mathbb R^2;\mathbb R^2)$.
\end{thm}

Theorems~\ref{thm:1}-\ref{thm:3} support the following conjecture.

\begin{conj}\label{conj:1}
Let $B\in C(\mathbb R^2)$. Then the energy space $\mathcal E_B$ is
non-empty if and only if $B={\rm curl}\, A'$ for some $A'\in
L^2(\mathbb R^2;\mathbb R^2)$.
\end{conj}

We finally conclude by mentioning  that we use two dimensional tools
in handling Theorems~\ref{thm:1}-\ref{thm:3}, that's why we could
not extend them to  three dimensions. However, as one may check
through the proofs, it still holds in three dimensions that the
energy space is empty when the magnetic field is constant.
Therefore, it sounds reasonable to believe that the results extend
to three dimensions as well.

\section{Preliminaries}
We start with some observations concerning the Ginzburg-Landau
equation in $\mathbb R^2$,
\begin{equation}\label{eq:GL-}
-(\nabla-iA)^2\psi=(1-|\psi|^2)\psi\quad{\rm in~}\mathbb
R^2\,.\end{equation}

\begin{lem}\label{lem:GL}
Assume that $A\in C^1(\mathbb R^2;\mathbb R^2)$. Let $\psi\in
C^2(\mathbb R^2;\mathbb C)$ be a solution of the Ginzburg-Landau
equation (\ref{eq:GL-}). Then $|\psi|\leq 1$.
\end{lem}
\begin{proof}
This is a classical consequence of the strong maximum principle, see
\cite[Chapter~3]{SaSe}.
\end{proof}

\begin{lem}\label{lem:lsc}
Let $A\in L^\infty_{\rm loc}(\mathbb R^2;\mathbb R^2)$. If the
energy space $\mathcal E_B$ is non-empty, then there exists a
finite-energy solution $\psi\in\mathcal E_B$ of the Ginzburg-Landau
equation (\ref{eq:GL-}).\\
If we assume in addition that $B\in C(\mathbb R^2)$, then up to a
gauge transformation, $\psi\in C^2(\mathbb R^2)$.
\end{lem}
\begin{proof}
The energy space being non-empty, we denote by
$$c_0=\inf_{\psi\in\mathcal E_B}E_B(\psi)\,.$$ We shall prove that $E_B$ admits a minimizer in $\mathcal E_B$.
To that end,  pick a minimizing sequence $(\psi_n)$ in $\mathcal
E_B$ such that
$$E_B(\psi_n)\to c_0\quad{\rm as~}n\to\infty\,.$$
Then, $\psi_n$ is pre-compact in $H^1(B(0,R);\mathbb C)$ for all
$R>0$. Consequently, using a standard diagonal argument, we may pick
a subsequence, still denoted by $(\psi_n)$, and a function $\psi\in
H^1_{\rm loc}(\mathbb R^2)$ such that
$$\psi_n\rightharpoonup\psi \quad{\rm weakly~in~}H^1(B(0,R);\mathbb
C)\,,\quad\forall~R>0\,.$$ By lower semi-continuity of the
$H^1$-norm, the continuous embedding of $H^1$ in $L^4$ and the
locally  compact embedding of   $H^1$ in $L^2$, it holds that,
\begin{eqnarray*}
&&\hskip-1cm\int_{B(0,R)}\left(|(\nabla-iA)
\psi|^2+\frac12(1-|\psi|^2)^2\right)\,\md x\\
&&\hskip0.5cm\leq \liminf_{n\to\infty}\int_{B(0,R)} \left(|(\nabla
-iA)\psi_n|^2+\frac12(1-|\psi_n|^2)^2\right)\,\md x \leq 2
c_0\,.\end{eqnarray*} The radius   $R>0$ being arbitrary, we deduce
that $E_B(\psi)\leq c_0$, hence $\psi\in\mathcal E_B$ and minimizes
$E_B$.
\end{proof}

Knowing more information about the magnetic potential $A$, we may
precise the behavior of finite-energy solutions of (\ref{eq:GL-}) as
$|x|\to\infty$.

\begin{lem}\label{lem:GL1}
Let $A\in  L^\infty(\mathbb R^2;\mathbb R^2)$ be such that  ${\rm
div}\,A\in L^\infty(\mathbb R^2)$ and ${\rm curl}\,A\in C(\mathbb
R^2)$. If $\psi$ is a finite-energy solution of (\ref{eq:GL-}), then
$1-|\psi|^2\in H^2(\mathbb R^2)$, hence $|\psi(x)|\to1$ as
$|x|\to\infty$.
\end{lem}
\begin{proof}
Setting $\varphi=1-|\psi|^2$, it is easy to establish that,
$$-\Delta\varphi+2\varphi=|(\nabla-iA)\varphi|^2+2\varphi^2\quad{\rm in~}\mathbb R^2\,.$$
Using the bound $|\psi|\leq 1$ of Lemma~\ref{lem:GL} and the fact
that $E_B(\psi)<\infty$, we infer that $\varphi\in H^1(\mathbb R^2)$
and $\varphi^2\in L^4(\mathbb R^2)$. By showing that
$|(\nabla-iA)\psi|\in L^4(\mathbb R^2)$, we invoke the $L^2$
regularity of $-\Delta+2$ and we deduce the desired result,
$\varphi\in H^2(\mathbb R^2)$.\\
So, let us establish that $|(\nabla-iA)\psi|\in L^4(\mathbb R^2)$.
Setting $v=(\nabla-iA)\psi$, we know that $v\in L^2(\mathbb R^2)$
since $\psi$ has finite energy. For instance, it holds that,
$$-(\nabla-iA)^2v=(1-|\psi|^2)v-2\psi|\psi|\,\nabla|\psi|\,.$$
Thus, using the diamagnetic inequality,
$|(\nabla-iA)\psi|\geq|\,\nabla|\psi|\,|$, the bounds $|\psi|\leq 1$
and $E_B(\psi)<\infty$, we deduce that $(\nabla-iA)^2v\in
L^2(\mathbb R^2)$.\\
Up to now, we have not used the hypotheses on $A$. We shall need
them to show that $\Delta v\in L^2(\mathbb R^2)$. Actually, it holds
that,
$$\Delta v=2iA\cdot\nabla v+i\left({\rm div}\,A\right)v-|A|^2v\,.$$
Since $v\in L^2(\mathbb R^2)$, we deduce that $\Delta v\in
L^2(\mathbb R^2)$. Consequently, we obtain  $v\in H^2(\mathbb R^2)$.
Using Sobolev embedding, we get what we desire to prove.
\end{proof}

\begin{lem}\label{lem:4}
Let $A\in C^1(\mathbb R^2;\mathbb R^2)$ be such that   $\nabla A\in
L^\infty(\mathbb R^2)$. If $\psi$ is a finite-energy solution of
(\ref{eq:GL-}), then $|\psi(x)|\to1$ as $|x|\to\infty$.
\end{lem}
\begin{proof}
Let us establish in a first step that $|(\nabla-iA)\psi|\in
L^\infty(\mathbb R^2)$. Assume by contradiction that there exists a
sequence $(x_n)$ such that
\begin{equation}\label{eq:contr}
|(\nabla-iA)\psi|(x_n)\to\infty\quad{\rm as}~n\to\infty\,.
\end{equation}
Define the translated functions,
$$\psi_n(x)=\psi(x_n+x)\,,\quad A_n(x)=A(x_n+x)\,.$$
Then, $\psi_n$ satisfies the following equation,
$$-\Delta \psi_n+2iA_n\cdot\nabla \psi_n+i({\rm div} A_n)\psi_n
=(1-|\psi_n|^2)\psi_n+|A_n|^2\psi_n\quad{\rm in~}\mathbb R^2.$$ Take
$R>0$, $p>2$ and let us establish the existence of positive
constants $C_R$, $C_{R,p}>0$ and a function $\chi_n\in H^1_{\rm
loc}(\mathbb R^2)$ such that, upon setting $A_n'=A-\nabla \chi_n$
and $\varphi_n=e^{i\chi_n}\psi_n$,
\begin{equation}\label{eq:contr1}
\|A'_n\|_{L^\infty(B_R)}\leq C_R\,,\quad
\|\varphi_n\|_{W^{2,p}(B_R)}\leq C_{R,p}\,,\quad\forall~n\in\mathbb
N\,.\end{equation} Once this is shown to hold, $\varphi_n$ becomes
bounded in $W^{2,p}(B_R)$, and hence, by the Sobolev embedding
theorem, in $C^{1,\alpha}(B_R)$ for any $\alpha\in(0,1)$. Since
$C^{1,\alpha}(B_R)$ is compactly embedded in $C^1(B_R)$, we get a
function $\varphi\in C^1(B_R)$ such that, upon extraction of a
subsequence, $\varphi_n$ converges to $\varphi$ locally in $C^1$.
Thanks again to (\ref{eq:contr1}), we get a constant vector
$a\in\mathbb R^2$ such that  by passing to a further subsequence,
$$|(\nabla-iA'_n)\varphi_n|(0)\to |(\nabla-i
a)\varphi|(0)\quad{\rm as~} n\to\infty\,.$$  Coming back to the
initial coordinates and gauge, this is in contradiction with
(\ref{eq:contr}).\\
Now we show why (\ref{eq:contr1}) holds. Actually, setting
$\chi_n(x)=A_n(0)x$, we get by the definition of $A'_n$ and the mean
value theorem,
$$|A'_n(x)|=|A_n(x)-A_n(0)|\leq \|\nabla
A_n\|_{L^\infty(B_R)}|x|\leq R\|\nabla A\|_{L^\infty(\mathbb
R^2)}\,,\quad\forall~x\in B_R\,.$$ The equation of $\varphi_n$
becomes,
$$-\Delta \varphi_n+2iA'_n\cdot\nabla \varphi_n+i({\rm div} A'_n)\varphi_n
=(1-|\varphi_n|^2)\varphi_n+|A'_n|^2\varphi_n\quad{\rm in~}\mathbb
R^2.$$ By Lemma~\ref{lem:GL}, $|\varphi_n|\leq 1$, hence there
exists a constant $C_R>0$ such that,
$$\|\Delta \varphi_n\|_{L^p(B_R)}\leq C_R+
2\|A'_n\|_{L^\infty(B_R)}\times\|\nabla
\varphi_n\|_{L^p(B_R)},\quad\forall~p\geq2\,.$$ Moreover, since
$\varphi_n$ has finite energy, we get by $L^2$ elliptic estimates
that $\varphi_n\in H^2(B_R)$. Using the embedding
$H^2(B_R)\hookrightarrow W^{1,p}(B_R)$ for all $p>2$, we conclude
through $L^p$ estimates that $\varphi_n\in W^{2,p}(B_R)$, proving
thus the desired bound in (\ref{eq:contr1}).\\
Now, having proved that $(\nabla-iA)\psi\in L^\infty(\mathbb R^2)$,
we deduce by the diamagnetic inequality that $\nabla|\psi|\in
L^\infty(\mathbb R^2)$. Therefore, $|\psi|$ is globally Lipschitz in
$\mathbb R^2$, and since $\psi$ has finite energy, $1-|\psi|^2\in
L^2(\mathbb R^2)$. This leads to the desired conclusion,
$1-|\psi(x)|^2\to 0$ as $|x|\to\infty$.
\end{proof}

 We close the section by recalling a result from the
spectral theory of magnetic Shr\"odinger operators.


\begin{lem}\label{lem:sh-op}
Under the assumptions of Theorem~\ref{thm:1}, there exists a
constants $C>0$ such that, for all $\psi\in H^1(\mathbb R^2;\mathbb
C)$ and $R>0$, the following inequality holds,
$$\int_{B(0,R)}|(\nabla-i A)\psi|^2\,\md x\geq
\frac12\int_{B(0,R/2)}B(x)|\psi|^2\,\md
x-\frac{C}{R^2}\int_{B(0,R)\setminus B(0,R/2)}|\psi(x)|^2\,\md
x\,.$$
\end{lem}
\begin{proof}
Let $\chi$ be a cut-off function such that $0\leq \chi\leq 1$,
$\chi=1$ in $[0,\frac12]$ and $\chi=0$ in $[1,\infty)$. Put
$$\chi_R(x)=\chi\left(\frac{|x|}{R}\right)\quad\forall~x\in\mathbb
R^2\,.$$ Next, we write,
\begin{eqnarray*}
\int_{B(0,R)}|(\nabla-iA)\psi|^2\,\md x&\geq&
\int_{B(0,R)}|\chi_R(\nabla-iA)\psi|^2\,\md x\\
&\geq&\frac12\int_{B(0,R)}|(\nabla-iA)(\chi_R\psi)|^2\,\md
x-\int_{B(0,R)}|\psi\nabla\chi_R|^2\,\md x\,.
\end{eqnarray*}
To finish the proof, we just use the following well known inequality
(see \cite{CFKS} or \cite[Lemma~2.4.1]{FH}),
$$\int_{B(0,R)}|(\nabla-iA)\phi|^2\,\md x\geq
\pm\int_{B(0,R)}B(x)|\phi|^2\,\md x\,,\quad\forall~\phi\in
H^1_0(B(0,R))\,.$$
\end{proof}

\section{Proof of main theorems}
\subsection{Proof of  Theorem~\ref{thm:1}}
\subsubsection{Necessary condition} Assume that the energy space
$\mathcal E_B$ is non-empty. Using Lemma~\ref{lem:lsc}, there exists
a solution $\psi\in\mathcal E_B$ of the Ginzburg-Landau equation
(\ref{eq:GL-}). Thanks to
Lemma~\ref{lem:GL}, we have the uniform estimate $|\psi|\leq1 $.\\
We would like to show that $B\in L^1(\mathbb R^2)$. To that end, it
is sufficient to bound $\displaystyle\int_{B(0,R)}B(x)\,\md x$
uniformly with respect to $R\in (1,\infty)$.\\
We therefore apply Lemma~\ref{lem:sh-op} (with $\psi$ as above, a
solution of (\ref{eq:GL-})). We get,
$$\int_{\mathbb
R^2}|(\nabla-iA)\psi|^2\,\md x\geq
\frac12\int_{B(0,R/2)}B(x)|\psi|^2\,\md x-\frac{C}{R^2}
\int_{B(0,R)\setminus B(0,R/2)}|\psi|^2\,\md x\,.$$ Using the bound
$|\psi|\leq 1$, we infer from the above estimate,
\begin{equation}\label{eq:s1}
\int_{\mathbb R^2}|(\nabla-iA)\psi|^2\,\md x\geq
\frac12\int_{B(0,R/2)}B(x)|\psi|^2\,\md x-\frac{3\pi
C}{4}\,.\end{equation} So, let
us handle the first term in the right hand side above.\\
We write,
$$\int_{B(0,R/2)}B(x)|\psi|^2\,\md x=
\int_{B(0,R/2)}B(x)\,\md x+\int_{B(0,R/2)}B(x)(|\psi|^2-1)\,\md
x\,.$$ Applying a Cauchy-Schwarz inequality, we get for all
$\varepsilon\in(0,1)$ (remark that $|\psi|^2-1\leq 0$),
$$\int_{B(0,R/2)}B(x)(|\psi|^2-1)\,\md x\geq
-\varepsilon\int_{B(0,R/2)}|B(x)|^2\,\md x-\varepsilon^{-1}
\int_{B(0,R/2)}(1-|\psi|^2)^2\,\md x\,.$$ Consequently, knowing that
$B$ is bounded and positive, we infer that
\begin{eqnarray*}
\int_{B(0,R/2)}B(x)|\psi|^2\,\md x&\geq&
\left(1-\varepsilon\|B\|_{L^\infty(\mathbb
R^2)}\right)\int_{B(0,R/2)}B(x)\,\md
x\\
&&-\varepsilon^{-1}\int_{B(0,R/2)}(1-|\psi|^2)^2\,\md
x\,.\end{eqnarray*}
 Choosing
$\varepsilon=\frac12(\|B\|_{L^\infty(\mathbb R^2)}+1)^{-1}$ and
replacing the above estimate in (\ref{eq:s1}), we deduce that,
$$
\frac14\int_{B(0,R/2)}B(x)\,\md x \leq
C'\left(E_B(\psi)+1\right)\,,\quad\forall~R\geq 1\,,$$ where
$C'=\max\left(1,\frac{3\pi C}{4},\frac12(\|B\|_{L^\infty(\mathbb
R^2)}+1)\right)$. Since the energy $\mathcal E_B(\psi)$ is finite,
we get the desired uniform bound.

\subsubsection{Sufficient condition} Assume now, in addition to the
hypotheses made in Theorem~\ref{thm:1}, that $B\in L^1(\mathbb
R^2)$. Then we get that $B\in L^p(\mathbb R^2)$ for all $p\geq1$.
Our aim next is to construct a magnetic potential $A'\in L^2(\mathbb
R^2;\mathbb R^2)$
such that ${\rm curl}\,A'=B$.\\
Define $\Gamma_2(x)=\frac1{2\pi}\ln|x|$, the fundamental solution of
the Laplacian  in two dimensions. Setting $w=\Gamma_2*B$, we get
$w\in  L^2(\mathbb R^2)$ (see \cite{GiTr}). Actually, taking $q\in
(1,2)$, we know that $\Gamma_2\in W^{1,q}(\mathbb R^2)$. Then using
Young's inequality,
$$\|\Gamma_2*B\|_{L^2(\mathbb R^2)}\leq \|\Gamma_2\|_{L^q(\mathbb
R^2)}\times \|B\|_{L^p(\mathbb
R^2)}\,,\quad\frac1p+\frac1q=\frac12+1\,,$$ we deduce that
$w=\Gamma_2*B\in L^2(\mathbb R^2)$.\\
Now, we observe that,
$$\Delta w=B\quad{\rm in~} \mathbb R^2\,,\quad w\in L^2(\mathbb R^2)\,,$$
from which we invoke $w\in H^2(\mathbb R^2)$. Let us now define the
magnetic potential $A'$ by $A'=\nabla ^\bot
w=(-\partial_{x_2}w,\partial_{x_1}w)$. Then $A'\in L^2(\mathbb R^2)$
and satisfies,
$${\rm curl}\, A'=B\,,\quad{\rm div}\,A'=0\quad{\rm in~}\mathbb
R^2\,,$$ which is what we desire to prove.

\subsection{Proof of Theorems~\ref{thm:2} and \ref{thm:3}}
The sufficient condition being immediate (see Remark~\ref{rem:1}),
we assume again that the energy space is non-empty, $\mathcal
E_B\not=\emptyset$. Therefore, by Lemma~\ref{lem:GL}, there exists a
solution $\psi$ of (\ref{eq:GL-}) such that $E_B(\psi)<\infty$.
Furthermore, $|\psi(x)|\to 1$ as $|x|\to \infty$. Actually, under
the hypotheses of Theorem~\ref{thm:2}, we use Lemma~\ref{lem:GL1},
and under those of Theorem~\ref{thm:3}, we use Lemma~\ref{lem:4}.\\
Now, up to a gauge transformation, we may assume that $\psi\in
C^2(\mathbb R^2)$. Thus, we may write $\psi=\rho e^{i\chi}$,
$\rho=|\psi|$, for a smooth real-valued function $\chi$. From
$E_B(\psi)<\infty$, we infer
$$\int_{\mathbb R^2}\rho^2|A-\nabla\chi|^2\,\md x<\infty\,.$$
Setting $A'=A-\nabla\chi$, we get that $A'\in L^2(\mathbb R^2)$ in
light of $\rho\to 1$ as $|x|\to\infty$.

\section*{Acknowledgements}
The author wishes to thank P. G\'erard for useful discussions.

\end{document}